\documentclass{l4dc2025}

\title[Kernel-Based Optimal Control: An Infinitesimal Generator Approach]{
Kernel-Based Optimal Control: An Infinitesimal Generator Approach
}
\usepackage{times}
\usepackage{user_macros}

\author{%
 \Name{Petar Bevanda}$^{1}$\Email{petar.bevanda@tum.de}\\
 \Name{Nicolas Hoischen}$^{1}$ \Email{nicolas.hoischen@tum.de}\\
\Name{Tobias Wittmann}$^{1}$ 
\Email{t.wittmann@tum.de}\\
\Name{Jan Br\"{u}digam}$^{1}$ \Email{jan.bruedigam@tum.de}\\
\Name{Sandra Hirche}$^{1}$ 
\Email{hirche@tum.de}\\
\Name{Boris Houska}$^{2}$ \Email{borish@shanghaitech.edu.cn} \\
$^{1}$\addr Chair of Information-oriented Control, Technical University of Munich, Germany\\
$^{2}$\addr School of Information Science and Technology, ShanghaiTech University
}

\begin{document}

\maketitle

\begin{abstract}%
This paper presents a novel operator-theoretic approach for optimal control of nonlinear stochastic systems within reproducing kernel Hilbert spaces. Our learning framework leverages data samples of system dynamics and stage cost functions, with only control penalties and constraints provided. The proposed method directly learns the infinitesimal generator of a controlled stochastic diffusion in an infinite-dimensional hypothesis space. We demonstrate that our approach seamlessly integrates with modern convex operator-theoretic Hamilton-Jacobi-Bellman recursions, enabling a data-driven solution to the optimal control problems. Furthermore, our learning framework includes nonparametric estimators for uncontrolled infinitesimal generators as a special case. Numerical experiments, ranging from synthetic differential equations to simulated robotic systems, showcase the advantages of our approach compared to both modern data-driven and classical nonlinear programming methods for optimal control.
\end{abstract}


\section{Introduction}

Traditional optimal control theory and algorithms rely on first principle models to represent system dynamics, with parameters learned from data through system identification techniques~\citep{Ljung1998}, either online or offline \citep{Rawlings2009}. These models enable the formulation of nonlinear optimal control problems, which can be solved to global optimality using existing methods, such as dynamic programming~\citep{beard1997galerkin}, based on Hamilton-Jacobi-Bellman (HJB) theory~\citep{Fleming1989,Crandall1992}, or other types of global optimal control methods~\citep{Houska2014}. However, due to their high complexity, these global optimal control methods are rarely used in practice~\citep{Houska2019}; at least not for nonlinear systems with many states and controls. Instead, nonlinear programming approaches, such as interior point and sequential quadratic programming methods, are commonly used to solve nonlinear optimal control problems locally \citep{howell2019altro,Biegler2007,Bock1984}. Although these local methods are sensitive to suitable initial guesses, they have reached a high level of maturity \citep{Diehl2002,Houska2011,Zavala2009}.

Recently, however, there has been a shift towards learning models directly from data. This shift leads to completely new types of system models as modern learning and system identification methods depart from relying on first principle models. While neural networks models are popular and often effective in terms of capturing the nonlinear behavior of control systems \citep{BeintemaPhD2024}, such representations are highly nonlinear and difficult to exploit for optimal control~\citep{lutter2020hjb, pmlr-v235-yang24f, Meng2024} due to modeling in a physical state-space.

In contrast to data-driven models in the physical state-space, modern operator-theoretic representations leverage linear operators, such as the Koopman operator \citep{mezic2004comparison,Mezic2005}, to represent the dynamics of observables (functions of the state) \citep{bevanda2024nonparametric, klus2020data, Kostic2022LearningSpaces, Vaidya2024, bruder2020data,Bevanda2023}, see \cite{Bevanda2021KoopmanControl,brunton2022modern} for recent reviews. These operator-theoretic models can be used as a basis to represent and learn system behavior from data. In the context of optimal control, the key advantages of operator-theoretic models is their compatibility with traditional stochastic control theory for diffusion processes \citep{Fleming1989, Crandall1992}. This combination allows for the convexification of nonlinear optimal control problems using infinite-dimensional representations \citep{houska2023convex,Vinter1993}. Furthermore, advances in synthesizing fully data-driven approaches for optimal control using kernel methods have shown promise in solving various types of optimal control problems~\citep{bevanda2024data,Gopalakrishnan2022,thorpe2022}.
\paragraph{Contribution.}
This article contributes novel theory and an innovative algorithm for data-driven nonlinear system identification and globally optimal stochastic control with key contributions in:
\begin{itemize}
    \item \textbf{System Identification.} We introduce a novel method for deriving non-parametric estimators of infinitesimal generators of controlled diffusions. Our approach leverages the properties of reproducing kernel Hilbert spaces (Lemma \ref{lem::KernelMatrix}) to learn the adjoint of the infinitesimal generator of strongly parabolic Fokker-Planck-Kolmogorov (FPK) equation.
    \item \textbf{Stochastic Optimal Control.} With the data-driven estimates of the infinitesimal operator dynamics we formulate a tractable continuous-time Kernel Hamilton-Jacobi-Bellman (KHJB) (Proposition~\ref{prop::HJB}) that enables the computation of approximations of globally optimal solutions to stochastic optimal control problems via Algorithm \ref{alg:IG-kHJB}. To demonstrate its performance, we apply this algorithm to both synthetic control systems and robotics benchmarks, showcasing its advantages over modern data-driven and classical nonlinear programming methods for optimal control.
\end{itemize}
Note that our estimators encompass nonparametric estimators for controlled infinitesimal generators of stochastic processes, which are notably absent from existing literature \citep{klus2016numerical,klus2020data,hou2023sparse,kostic2024learningGenerator}. This distinction sets our work apart from finite-dimensional deterministic settings \citep{buzhardt2023controlled,nuske2023finite} and unforced transfer operators \citep{Kostic2022LearningSpaces, kostic2024consistent}. Moreover, the continuous-time nature of our approach reduces the dependence on a specific time-lag of the data. This allows for the derivation of explicit controllers and value functions, and provides a counterpart to the discrete-time learning-based optimal control method presented in \citep{bevanda2024data}. 
\\{\textbf{Structure}\footnote{\textbf{Notation.} 
 Let $\mu {\in} M_+(\mathbb{X})$ be a measure with full support, $\mathrm{supp}(\mu) {=} \mathbb{X}$ and $M(\mathbb{X})$ the corresponding set of bounded signed Borel measures. The symbols $C^k(\mathbb{X})$, $L_\mu^k(\mathbb{X})$, $H_\mu^k(\mathbb{X})$ represent the set of $k$-times continuously differentiable, $L_\mu^k$-integrable, and $k$-times weakly differentiable functions with $L_{\mu}^2$-integrable derivatives. For non-negative integers $n$ and $m$, $[m,n]=\{m,m{+}1,\dots,n\}$ with $n\geq m$ gives an interval set $[n]\defeq[1,n]$. Given a separable Hilbert space $\mathcal{H}$ we let $\HS{\mathcal{H}}$ be a Hilbert space of Hilbert-Schmidt (HS) operators from $\mathcal{H}$ to itself with norm $\hnorm{A}^2\equiv \lilsum_{i\in\Set{N}}\norm{Ae_i}^2_{\mathcal{G}}$ where $\{e_i\}_{i\in\Set{N}}$ is an orthonormal basis of $\RKHS$.}. After the problem statement in Section \ref{sec:PbStatement}, we present an equivalent convex optimal control formulation, using controlled Fokker-Plank-Kolmogorov (FPK) equations. Section \ref{sec:LearningGeneartor} derives a novel operator regression in RKHS to approximate the infinitesimal generator for controlled diffusion, yielding feedback control policies through a simple dynamic programming recursion. Finally, Section \ref{sec:NumericalResults} demonstrates our approach on a robotic swing-up task on the inverted pendulum and cartpole and is validated by benchmark examples for synthetic ODEs and aforementioned simulated nonlinear systems.} 
\section{Problem Statement} \label{sec:PbStatement}
This work aims to learn an optimal feedback policy $\bm{\pi}^\star$ for a control-affine nonlinear system, such that given the current state $\bx$, the input $\bm{u}^\star = \bm{\pi}^\star(\bx)$ solves the infinite horizon optimal control problem
\begin{equation}
    \displaystyle{\minimize_{\bm{x, u}} \int_{0}^{\infty}  (q(\bx) + r(\bm {u}))} \: \text{d}t \ \hspace{5pt} \text{s.t.} \hspace{5pt} \dot{\bx} = \bm{f}(\bm{x}) + \bm{G}(\bm{x}) \bm{u}=:\bm{f}_{\bm{u}}(\bm{x}), 
    \hspace{5pt}
    \bm{u} \in \mathbb{U},
    \label{eq:problem_statement_cost}
\end{equation}
with $\bm{f} \in C^1(\mathbb{X})^{n_x}$, {$\bm{G} \in C^1(\mathbb{X})^{n_x \times n_u}$} and {stage} cost $\stgcost\in C^1(\mathbb{X})$. Here, we define $\mathbb X \defeq \mathbb R^{n_x}$. Moreover, $\mathbb{U} \defeq \{\bm{u} \in \mathbb{R}^{n_u} \mid \bm{u}_{-} \leq \bm{u} \leq \bm{u}_{+} \}$ denotes control bounds and $r \in C^1(\mathbb U)$ is a strongly convex control penalty. Real systems, commonly modeled by ODEs, are often subject to process noise, which can be taken into account by replacing the deterministic
models by a stochastic differential equation (SDE)~\citep{oksendal2013stochastic}. 
To that end, we build on the approach described in \citep{houska2023convex, bevanda2024data} by considering the extension of \eqref{eq:problem_statement_cost} to dynamics subject to a small white noise disturbance, leading to the closed-loop process
\begin{equation}
\mathrm{d}\bm{X}_t = \left( \bm{f}(\bm{X}_t)+\bm{G}(\bm{X}_t) \bm{\pi}( \bm{X}_t ) \right) \mathrm{d}t + \sqrt{2 \epsilon}~\mathrm{d} \bm{W}_t, \label{eq::ctrlSDE} \tag{\txt{c\textsc{sde}}}
\end{equation}
where $\bm{W}_t$ is a $\R^{n_x}$ Wiener process and $\epsilon > 0$ a diffusion parameter, modeling the amplitude of the process noise. Our goal is to obtain a (Lebesgue measurable) feedback $\bm{\pi}: \spX \to \mathbb U$ that minimizes the average \emph{ergodic} cost,
\begin{eqnarray}
\label{eq::SOCP}
\lim_{T \to \infty} \ \min_{\bm{\pi}:\mathbb{X} \mapsto \mathbb U} \  \expect \left[ \frac{1}{T} \int_0^T  \left( \stgcost(\bm{X}_t) + r( \bm{\pi}(\bm{X}_t) ) \right) \, \mathrm{d}t \right] \hspace{10pt} \text{s.t.} \hspace{5pt} \eqref{eq::ctrlSDE}.
\end{eqnarray}
The above stochastic optimal control problem formulation can {either} be viewed both as a
\textit{viscosity solution to}~\eqref{eq:problem_statement_cost} {for $\epsilon \to 0^+$, whenever this limit exists, or as} an \textit{effort to identify more robust control policies for noisy or uncertain dynamics} {for $\epsilon > 0$}. To ensure the above infinite-horizon problem is well-defined, we require the following assumption from \citep{bevanda2024data}.
\begin{assumption}\label{ass:GC}
There exists a $\bm{\pi}: \mathbb{X} \to \mathrm{int}(\mathbb U)$ with $\bm{\pi} \in L^\infty(\mathbb{X})$ and a strongly convex $\mathcal V \in C^2(\mathbb{X})$ with bounded Hessian, and constants $0 < c_1,c_2 < \infty$ such that
$
(\bm{f}(\bm{x})+\bm{G}(\bm{x}) \bm{\pi}(\bm{x}))^* \nabla \mathcal V(\bm{x}) \leq c_1 - c_2 ( \stgcost(\bm{x}) + r(\bm{\pi}(\bm{x})) ),
$
for all $\bx \in \mathbb{X}$. The set $\mathbb U \subseteq \mathbb R^{n_u}$ is closed, convex, and $\mathrm{int}(\mathbb U) \neq \varnothing$. The control penalty $r \in C^1(\mathbb U)$ is strongly convex.
\end{assumption}
The above condition is often met in practical scenarios; see~\cite {houska2023convex,bevanda2024data} for an in-depth discussion. We require the following dataset to learn~\eqref{eq::ctrlSDE}.
\begin{assumption}\label{asm:data}
There are state observations $\mathsf{X}\defeq\{ \bxi \}_{i\in[N]}$ of the nominal system $\dot{\bm{x}} = \bm{f}(\bm{x})+\bm{G}(\bm{x}) \bm{u}$ under no excitation $\bm{u}_0\defeq\bm{0}$ and under “one-hot” (standard basis) input vectors $\{\bm{u}_j\defeq\bm{e}_j\}_{j \in [n_u]}$, forming a dataset 
 \begin{align}\label{eq:data} 
    \Set{D}^N=\{\Set{D}^N_j\}^{n_u}_{j=0} \qquad \text{where} \qquad \Set{D}^N_j\defeq \left\{\bxi, \dot{\bm{x}}^{(i)}_{\bm{u}_j}\defeq \bm{f}_{\bm{u}_j}({\bm{x}}^{(i)})\right\}_{i=1}^{N}.
\end{align}
\end{assumption}
While the above assumption may be restrictive for real-world data collection, it is still readily fulfilled in many settings, for example, when using existing physical or data-driven models \citep{umlauft2017learning} as well as data from high-fidelity simulators \citep{howelllecleach2022dojo}. 
Depending on whether the stage cost $q$ is known or not, the vector $\bm{q}_\SFX=[q(\bxi)]_{i\in[N]}$ can subsequently be measured or computed on $\SFX$.
 After a reformulation of a convex optimal control problem equivalent to~\eqref{eq::SOCP}, we propose learning infinitesimal generators of the associated controlled diffusion process, leading to a simple kernel-based dynamic programming recursion using the dataset \eqref{eq:data}.
\section{Operator-Theoretic Dynamic Programming and HJB Recursions}
As long as Assumption~\ref{ass:GC} holds, a martingale solution $\Xt$ to~\eqref{eq::ctrlSDE} exists for at least one feasible feedback $\bm{\pi}$; see~\cite{houska2023convex} and \citep[Section II]{bevanda2024data}. Moreover, under the additional assumption that the stage cost of~\eqref{eq::SOCP} has a bounded variance for the optimal ergodic limit distribution, existence of an optimal ergodic solution can be guaranteed for any $\epsilon > 0$; see~\citep[Thm.~1]{houska2023convex}. The solution $\Xt$ constitutes a time-homogeneous Markov diffusion process whose transition operator, $\Gamma_{\bm{\pi}}(t)$, maps the probability density function $\rho_0 \in D_+(\spX)$ of $\randBx_0$ to the probability density function $\rho_t \in H^1(\spX)$ of $\Xt$ so that $\rho_t = \Gamma_{\bm{\pi}}(t) \rho_0$, where $D_+(\mathbb{X})$ is the set of non-negative bounded distributions on $\mathbb{X}$.
It is well-known~\citep{Hinze2009,Oksendal2000} that---under mild regularity assumptions---$\Gamma_{\bm{\pi}}(t)$ is for any given $\bm{\pi}\in L^\infty((0,T)\times \spX)^{n_u}$ a bounded linear operator on {$H^1(\spX)$}.

It is worth recalling that the infinitesimal generator {$\fFPK: H^1(\spX) \to H^1(\spX)$} associated to the Perron-Frobenius operator,
is defined for every {$\rho \in H^1(\spX) $} such that the limit $\fFPK \rho\ \defeq \textstyle \lim_{t \to 0^+} \frac{\Gamma_{\bm{\pi}}(t) \rho - \rho}{t}$ exists under mild regularity assumptions \citep{engel2000one}. {The map} $\fFPK$ is called \textit{Fokker-Planck-Kolmogorov (FPK)} operator~\citep{Bogachev2015} or Kolmogorov's infinitesimal generator  \citep{froyland2023detecting}. It describes the evolution of the probability density of the state of a diffusion process $\Xt$ under arbitrary feedback laws ${\bm{\pi}}\in L^\infty(\mathbb{X})$. 
Notice that the backward FPK operator $\bFPK$, relates to the infinitesimal generator of the stochastic Koopman operator for autonomous stochastic differential equations (SDEs) \citep{kostic2024learningGenerator} that describes the evolution of observables {$h \in H^1(\mathbb X)$  under $\Xt$}.
Observing that $\fFPK$ is affine in $\bm{\pi}$, will allow us to reformulate~\eqref{eq::SOCP} as an equivalent convex optimization problem \citep{houska2023convex}.
\subsection{Convex Reformulation}
\label{sec::ConvexReformulation}
If $\epsilon > 0$, we have $\rho_t = \Gamma_{\bm{\pi}}(t) \rho_0 > 0$ for $t > 0$, since $\rho_t$ is the state of a uniformly parabolic diffusion process \citep{Bogachev2015}. Exploiting that $\fFPK$ is affine in $\bm{\pi}$ allows to write the infinitesimal generator in the form  
\begin{align}
\label{eq:fFPK_rho}
  \textstyle  \dot{\rho} \ \defeq \fFPK \rho = \mathcal{A} \rho + \mathcal{B} (\bm{\pi} \rho)
\end{align}
with linear-system shorthands $\mathcal{A}\rho = - \nabla \cdot \left(\bm{f} \rho  \right) + \epsilon \nabla^2 \rho \quad \text{and} \quad \mathcal{B}(\bm{\pi}\rho) = - \nabla \cdot \left(  \bm{G} \bm{\pi} \rho  \right)$, with the autonomous $\mathcal{A}: H^1(\Set{X}) \to H^1(\Set{X})$, and the control  $\mathcal{B}: [H^1(\Set{X})]^{n_u} \to H^1(\Set{X})$ FPK operator.
Then, the {ergodic} optimal control problem \eqref{eq::SOCP} can equivalently be cast as a convex PDE-constrained optimization problem through a change of variables, namely, $\bm{\nu} = \bm{\pi}{\rho}$ \citep{houska2023convex}. 
In particular, for a finite time horizon $T < \infty$, and a given initial probability distribution of the state, $\rho(0) = \rho_0  ~\text{on}~\spX$, the finite-horizon stochastic optimal control problem is equivalent to solving the PDE-constrained convex optimization problem
\begin{align}
\label{eq::PDEOCP}
\!\!\!\!\!\!\!\!\!\mathcal J(T,\rho_0) {\defeq} \min_{\rho,\bm{\nu}} \int_0^T \int_{\mathbb{X}} \left(\stgcost{+}r\left(\frac{\bm{\nu}}{\rho}\right) \right) \rho \, \mathrm{d}\bx \, \mathrm{d}t \quad \mathrm{s.t.}~
\left\{
\begin{array}{l}
\dot \rho {=}  \mathcal{A} \rho {+} \mathcal{B}\bm{\nu} ~\text{on}~(0,T) {\times} \mathbb{X}\\
\rho(0) {=} \rho_0  ~\text{on}~ \mathbb{X}\\
\bm{\nu} \in \rho \Set{U}
\end{array}
\right.
\tag{\txt{\textsc{pde-ocp}}}
\end{align}
Under the mentioned regularity assumptions, the ergodic limit in~\eqref{eq::SOCP} is given by $\lim_{T \to \infty} \ \frac{1}{T} \ \mathcal J(T,\rho)$ and invariant for $\epsilon  > 0$ in the sense that it does not depend on the initial distribution $\rho_0$, see~\cite[Thm.~1]{houska2023convex}.
\paragraph{Hamilton-Jacobi-Bellman equations}
To exploit the convex duality of HJBs and FPKs, we first introduce the Fenchel conjugate $r^*( \bm{\lambda} ) \defeq  \sup_{\bm{u} \in \mathbb U}\left\{  \langle\bm{\lambda}, \bm{u}\rangle  - r(\bm{u}) \right\}$ of the control penalty ${r}$ and define the following optimal policy form $\bm{u}^\star( \cdot)$ and dual function $\mathcal{D}_r(\cdot)\equiv -r^*(-(\cdot)) $, respectively
\begin{align}
\bm{u}^\star( \bm{\lambda}) \defeq {\argmin_{\bm{u} \in \mathbb U}}\left\{ r(\bm{u}) + \langle\bm{\lambda}, \bm{u}\rangle \right\} \quad  ~\text{and}~ \quad 
\mathcal{D}_r( \bm{\lambda} ) \defeq \min_{\bm{u} \in \mathbb U}\left\{ r(\bm{u}) + \langle\bm{\lambda}, \bm{u}\rangle \right\}.
\end{align}
so $\bm{u}^\star( \bm{\lambda}) {=}{-} \nabla r^*(\bm{\lambda })$ and $
\mathcal{D}_r( \bm{\lambda} ) {=} r(\bm{u}^\star( \bm{\lambda} )){+} \langle\bm{\lambda}, \bm{u}^\star(\bm{\lambda})\rangle$.
We construct an associated dual problem with the {strongly measurable} functional ${V}{:} [0,T] {\to} H_\mu^1(\mathbb{X})$ as a co-state, such that the infinite-dimensional optimization problem \eqref{eq::PDEOCP} can be solved using the stochastic HJB equation
\begin{align}
\label{eq:HJB}
-\dot V =  \mathcal{A}^* V +   q +  \mathcal{D}_r( \mathcal{B}^* V ) \quad \mathrm{s.t.}\quad V(T) = 0  \qquad ~\text{on}~ \Set{X}_T
\tag{\txt{\textsc{hjb-fvp}}}
\end{align}
which can be interpreted as a \textit{final value problem} (FVP) in $H_\mu^1$.

The associated initial value $V(0)$ turns out to be a Riesz representation of the cost functional of~\eqref{eq::PDEOCP}; that is, we have $\mathcal J(T,\rho_0) = \langle \rho_0,V\rangle$ for all initial probability distributions $\rho_0 \in H_{\mu^{-1}}^1(\mathbb X)$.
In this context, one needs to introduce a suitable ergodic probability measure $\mu \in H^1(\mathbb X)$ in order to define the weighted Sobolev space $H_\mu^1$;
see~\citep{houska2023convex} for details. Conditions for which this argument holds (i.e., under strong duality) are nonrestrictive and can be found in \citep[Theorem 2]{bevanda2024data}.
As the above stochastic HJB provides a Hilbert space FVP, a space discretization is usually required for practical computation, e.g., one coming from Galerkin methods as proposed in~\citep{houska2023convex}. Such discretizations are, however, computationally demanding in moderate to high dimensional state spaces. In the following section, we devise a flexible (nonparametric) framework and provide a data-driven approximation of \eqref{eq:HJB} based on \eqref{eq:data}.
\section{Generator Regression for Control-Affine Diffusions and HJB Approximation} \label{sec:LearningGeneartor}
\paragraph{Reproducing Kernel Hilbert Spaces}
We require the RKHS $\RKHS$ to be a dense subspace of the classical Sobolev space $H^1(\spX)$. 
Additionally, we recall that an RKHS is associated with a kernel function $k: \spX \times \spX \to \mathbb R$ that is symmetric positive definite. Let $\phi(\bx)=k(\cdot,\bm{x}): \spX \to \RKHS$ be the canonical feature maps, denoted by $\phi(\bx)$, which can be assumed to satisfy $\phi(\bx) \in H^1(\spX)$ for all $\bx \in \spX$. 
 Moreover, $\forall \bx, \bx' \in \spX$, we have that $k(\bx, \bx') = \scalarpH{\phi(\bx),\phi({\bx'})}=\scalarpH{k(\cdot,\bx),k(\cdot,{\bx'})}$ and the reproducing property $h(\bm{x}) = \langle h,k(\cdot,\bx)\rangle_{\RKHS}$ holds for all $\bx \in \spX$ and all observables $h \in \RKHS$. We introduce the (canonical) \textit{embedding operator}, given for any $f \in L^2$ as $\embedIN:f\mapsto \textstyle \expect_{\randBx\sim p}[f(\randBx) \phi(\randBx)]$, where $p \in \mathcal{M}_+(\mathbb{X})$ a probability measure. Its adjoint, the inclusion operator $\injectIN: \RKHS \to  L^2$, is given by $(\injectIN f)(\bx) = f(\bx)$, for all $\bx \in \spX$ which we consider to be a Hilbert-Schmidt operator, which holds under very mild technical conditions \citep{IngoSteinwart2008SupportMachines}. 
 As we are dealing with differential operators, we require the following assumption for well-posed regression.
 \begin{assumption}\label{asm:RKHS}
 The RKHS $\RKHS$ is norm-equivalent to \(H^{s}(\Set{X})\) with $s > \frac{n_x}{2}+1$
 \end{assumption}
 This condition guarantees, via \textit{Maurin's Theorem} \citep{adams2003sobolev}, that the canonical inclusion $\RKHS \hookrightarrow H^1(\Set{X})$ is a Hilbert-Schmidt (HS) operator, allowing us to formulate \textit{Hilbert-Schmidt regression} problems \citep{Mollenhauer2022} to learn differential operators.\footnote{The requirement stems from the condition $s > n_x/2 +1$ for an embedding $H^s \hookrightarrow H^1$ to be Hilbert-Schmidt. The above assumption is fulfilled by popular kernels with sufficiently regular RKHS, e.g., Matern and Gaussian kernels.}

\paragraph{Hilbert-Schmidt regression in infinite dimensions} 
 To approximate the operators $\mathcal{A}^*,\mathcal{B}^*$ in a nonparametric manner, we look for RKHS approximations $\{L^*_{\bm{\pi}}: \RKHS \to \RKHS\}_{{\bm{\pi}~\in~\bm{0} \cup \{\bm{e}_j\}_{j \in [n_u]}}}$ based on its RKHS restrictions $\{\bFPK\!\mid_\RKHS: \RKHS \to L^2\}_{{\bm{\pi}~\in~\bm{0} \cup \{\bm{e}_j\}_{j \in [n_u]}}}$.
 This is a valid strategy, as the control affinity of \eqref{eq::ctrlSDE} is inherited by its infinitesimal generator so that
 \begin{align}
\mathcal{A}^* \equiv \mathcal{L}^*_{\bm{0}}\qquad \text{and}\qquad \mathcal{B}^* \equiv [(\mathcal{L}^*_{\bm{e}_1}-\mathcal{L}^*_{\bm{0}}) \cdots (\mathcal{L}^*_{\bm{e}_{n_u}}-\mathcal{L}^*_{\bm{0}})],     
 \end{align}
 we define the following set of risk functionals
\begin{align}\label{eq:risk}
   \big\{\mathcal{R}(L^*_{\bm{\pi}})=\|\bFPK\!\mid_\RKHS-\injectIN L^*_{\bm{\pi}}\|^2_{\HS{\RKHS,L^2}}\big\}_{\bm{\pi}~\in~\bm{0} \cup \{\bm{e}_j\}_{j \in [n_u]}},
\end{align}
to measure the mean square error for regression given the labels $\{\bFPK\!\mid_\RKHS\}_{{\bm{\pi}~\in~\bm{0} \cup \{\bm{e}_j\}_{j \in [n_u]}}}$. The latter is classical in the context of infinite-dimensional regression \citep{Kostic2022LearningSpaces,Kostic2023KoopmanEigenvalues,Li2022opLern,Mollenhauer2022,mollenhauer2020nonparametric}. Based on mild regularity conditions \citep{houska2023convex} the  FPK operator $\fFPK$ admits a bounded adjoint $\bFPK$ on $H^1\subseteq L^2$. Then, by selecting a suitably regular RKHS to satisfy Assumption \ref{asm:RKHS}, and ensuring that the restriction $\bFPK\!\mid_\RKHS \, \in \HS{\RKHS, L^2}$, the optimization objective in \eqref{eq:risk} becomes well-defined.

Still, in practice, minimizing \eqref{eq:risk} may require solving a badly conditioned equation system. Thus, we formulate a Tikhonov-regularized problem
 \begin{align}\label{eq:regfFPKonH}
\Big\{L_{\bm{\pi}}^* \defeq  \argmin_{\HfFPK  \in \mathrm{HS}(\RKHS)} \mathcal{R}(\HfFPK) + \gamma\| \HfFPK \|^2_{\mathrm{HS}}=  C_\gamma^{-1} {T_{\bm{\pi}}}\Big\}_{\bm{\pi}~\in~\bm{0} \cup \{\bm{e}_j\}_{j \in [n_u]}}
  , \quad \mathsf{\gamma}> 0 
  \tag{\txt{\textsc{krr-fpk}}}
 \end{align}
which corresponds to the \textit{Kernel Ridge Regression} (KRR) approximation of $\fFPK$ over $\RKHS$ where operators $T_{\bm{\pi}}, C$ are defined as
\begin{align}\label{eq:Cov}
    C \defeq \embedIN\injectIN = \expect \left[\phi({\randBx}) \otimes \phi({\randBx}) \right] \qquad 
     \big\{ T_{\bm{\pi}}  \defeq  \embedIN \embedIN^*_{\bm{\pi}}= \expect \left[\phi({\randBx})\otimes \target({\randBx}) \right] \big\}_{\bm{\pi}\in\bm{0} \cup \{\bm{e}_j\}_{j \in [n_u]}},
\end{align}
where $\embedIN^*_{\bm{\pi}}= \bFPK\!\mid_\RKHS$.
with the \textit{regularized covariance} $C_\gamma=C+\gamma I_{\RKHS}$. In practice, the risk \eqref{eq:risk} can only be evaluated on data \eqref{eq:data} leading to the \textit{empirical risk minimizations}
\begin{align}\label{eq:regfFPKonHestim}
    \Big\{{\widehat{L}^*_{\bm{\pi}}} \defeq  \argmin_{\HfFPK  \in \mathrm{HS}(\RKHS)} \widehat{\mathcal{R}}(\HfFPK)+\gamma\| \HfFPK\|^2_{\mathrm{HS}} = \widehat{C}_\gamma^{-1}{\widehat{T}_{\bm{\pi}}}= \SembedIN\bm{K}^{-1}_\gamma  \SinjectOUT_{\bm{\pi}}\Big\}_{\bm{\pi}~\in~\bm{0} \cup \{\bm{e}_j\}_{j \in [n_u]}} \tag{$\widehat{\txt{\textsc{krr-fpk}}}$}
\end{align}   
to obtain a finite rank operators $\widehat{L}^*_{\bm{0}}\cup\{\widehat{L}^*_{\bm{e}_j}\}_{j \in [n_u]}$. This is due to the Gram matrix $ \bm{K}\defeq \SinjectIN\SembedIN=[k(\bxi,\bxj)]_{i,j \in [N]}$ in $\bm{K}^{-1}_\gamma \defeq (\bm{K} + N\gamma I_N)^{-1}$ uncovered by the \textit{Sherman-Morrison-Woodbury} formula and the reproducing property. The \textit{data-based injection} operators are defined $\SinjectOUT,\SinjectOUT_{\bm{\pi}} \in \mathrm{HS}(\RKHS, \R^N)$ as
\begin{align}\label{eq:Sampl}
    \SinjectIN:  \RKHS \to \mathbb{R}^N ~ \mathrm{s.t.}~ \phi \mapsto [\phi^*(\bm{x}^{(i)})]_{i \in [N]}  \qquad \SinjectOUT_{\bm{\pi}}: \RKHS \to \mathbb{R}^N ~ \mathrm{s.t.}~ \phi \mapsto [\target^*(\bm{x}^{(i)})]_{i \in [N]}.
\end{align}
and their adjoints, \textit{data-based embeddings},
\begin{align}\label{eq:Eembed}
    \SembedIN: \mathbb{R}^N  \to  \RKHS ~ \mathrm{s.t.}~ \bm{w} \mapsto  \textstyle \sum_{i\in [N]}{w}_i\phi(\bm{x}_i)  \qquad \SembedOUT_{\bm{\pi}}: \mathbb{R}^N  \to  \RKHS   ~ \mathrm{s.t.}~ \bm{w} \mapsto  \textstyle \sum_{i\in [N]}{w}_i\target(\bm{x}_i),
\end{align}
which can informally be considered as kernel-induced feature matrices. While \eqref{eq:regfFPKonHestim} is defined over the RKHS $\RKHS$, in practice, our computations will only require the Gram matrix $\bm{K}$ and \textit{target kernel matrices} $\bm{K}_{\bm{0}_{~}}{\defeq} \SinjectIN_{\bm{0}_{~}}\SembedIN,\bm{K}_{\bm{e}_j}{\defeq} \SinjectIN_{\bm{e}_j}\SembedIN \in \R^{N\times N}$ for any $j {\in} [n_u]$ by plugging in the data \eqref{eq:data} from Assumption \ref{asm:data}, derived in the following lemma.
\begin{lemma}
\label{lem::KernelMatrix}
Let $k$ be a Mercer kernel such that $k \in C^4(\Set{X} \times \Set{X})$ with corresponding RKHS $\RKHS$ and the system dynamics be described by \eqref{eq::ctrlSDE} under inputs $\bm{\pi}~\in~\bm{0} \cup \{\bm{e}_j\}_{j \in [n_u]}$. Then, the entries of the target kernel matrices are computed via
\begin{align}\label{eq:kernelTarget}
   \left(\Kyx_{\bm{\pi}}\right)_{ij}  {=}  \langle\textstyle (\bm{f}(\bxi) {+} \bm{G}(\bxi) \bm{\pi}(\bxi)),\textstyle \nabla_{\bxi} k(\bxi, \bxj) \rangle  {+} \epsilon \mathrm{Tr}(\nabla^2_{\bxi}k(\bxi, \bxj)).
\end{align}
\end{lemma}
\begin{proof}
First we use It$\overline{\text{o}}$ formula \citep{arnold1974stochastic,kostic2024learningGenerator}
associated to \eqref{eq::ctrlSDE} for $\bm{u}\equiv \bm{\pi}(\bx)$ to compute $\SembedIN_{\bm{\pi}}$. After using \eqref{eq:Sampl}-\eqref{eq:Eembed} we have 
$\left(\Kyx_{\bm{\pi}}\right)_{ij}{\defeq}(\SinjectOUT_{\bm{\pi}}\SembedIN)_{ij}=\scalarpH{\target(\bxj),\phi(\bxi)}$ so \eqref{eq:kernelTarget} is obtained after applying the \textit{derivative reproducing property} \citep[Theorem 1]{ZHOU2008456}.
\end{proof}
With the control-affinity of the dynamics inherited by \eqref{eq:kernelTarget} for $\Kyx_{\bm{\pi}}$, we obtain matrices
    \begin{align}\label{eq:FinDim}
    \SFA & {\defeq} \bm{K}^{-1}_\gamma \Kyx_{\bm{0}_{~}}{ \in} \R^{N\times N}
   \qquad \text{and}\qquad \SFB {=} [\SFB_1 \cdots \SFB_{n_u}], \SFB_i {\defeq} \bm{K}^{-1}_\gamma(\Kyx_{\bm{e}_i}{-}\Kyx_{\bm{0}})\in \R^{N\times N},
    \end{align}
    that are fully described using data from Assumption \ref{asm:data} after setting $\bm{\pi}$ to $\{\bm{u}_j\}^{n_u}_{j=0}$ in \eqref{eq:kernelTarget}. With the help of the above lemma and control system matrices \eqref{eq:FinDim}, we can state the following result.
\begin{proposition}
\label{prop::HJB}
Let the estimates for \eqref{eq:HJB} be $\hat{A}\defeq \SembedIN\bm{K}^{-1}_\gamma\SinjectOUT_{{0}}$, $\{\hat{B}_i\defeq  \SembedIN\bm{K}^{-1}_\gamma (\SinjectOUT_{\bm{e}_m} {-} \SinjectOUT_{{0}})\}_{m=1}^{n_u}$, $\hat{\dot{V}} {\defeq}  \SembedIN \dot{\bm{v}},\hat{V} {\defeq}  \SembedIN\bm{v}, \hat{q} {\defeq}  \SembedIN\bm{q}, \hat{D}_r(\hat{B}\hat{V}) {=} \SembedIN \bm{K}^{-1}_\gamma \SinjectIN \mathcal{D}_r(\hat{B}\hat{V}) \in \RKHS$, where $\RKHS$ fulfills the conditions of Lemma \ref{lem::KernelMatrix}. Then, solving the infinite-dimensional
$$
\scalarp{- \dot{\hat{V}}= \hat{A}  \hat{V}+\hat{q} + \hat{D}_r(\hat{B}\hat{V}), k_{\bm{x}}} \qquad \mathrm{s.t.} \qquad \scalarp{\hat{V}_T{=}{0},k_{\bm{x}}} \quad  ~\text{on}~~(0,T) {\times} \mathbb{X}.
$$ amounts to $N$-dimensional final-value problem (FVP)
\begin{align}\label{eq:finiteHJB}\tag{$\mathsf{\widehat{\txt{\textsc{hjb-fvp}}}}$}
\scalarp{-\dot{\bm{v}}= \SFA  \bm{v}+\bm{q} + \bm{D}_r(\SFB \bm{v}), \bm{k}(\bx)} \quad \mathrm{s.t.} \quad \scalarp{\bm{v}_T{=}\bm{0},\bm{k}(\bx)}\quad~~\text{on}~~(0,T) {\times} \mathbb{X}. 
\end{align}
where $\bm{k}(\bx)=[k(\bxi, \bx)]_{i \in [N]} \in \R^{N}$ and $\bm{q}=\bm{K}^{-1}_\gamma \bm{q}_\SFX$, $\bm{D}_r(\bm{B}\bm{v})=\bm{K}^{-1}_\gamma [\mathcal{D}_r(\scalarp{\bm{{B}}{\bm{v}},\bm{k}(\bm{x}^{(i)})})]_{i\in [N]}$.
\end{proposition}
    \begin{proof}
The estimated \eqref{eq:HJB} takes the following form in $\RKHS$
$$
\scalarp{- \dot{\hat{V}}= \hat{A}  \hat{V}+\hat{q} + \hat{D}_r(\hat{B}\hat{V}), k_{\bm{x}}} \qquad \mathrm{s.t.} \qquad \scalarp{\hat{V}_T{=}{0},k_{\bm{x}}} ~~\text{on}~~(0,T) {\times} \mathbb{X}.
$$
After plugging the defined estimates, we obtain
\begin{subequations}
    \begin{align} 
    \langle-\SembedIN\dot{\bm{v}}&{=}\SembedIN \bm{K}^{-1}_\gamma  \SinjectOUT_{{0}} \SembedIN\bm{v}+\SembedIN \bm{q}+ \SembedIN \bm{K}^{-1}_\gamma \SinjectIN \mathcal{D}_r(\hat{B}\hat{V}), k_{\bx}\rangle ~~~ \mathrm{s.t.} ~~~ \scalarp{\bm{v}_T{=}\bm{0},\bm{k}(\bx)}~~\text{on}~~(0,T) {\times} \mathbb{X},\label{eq:PreTrick}\\
    \langle-\dot{\bm{v}}&{=} \SFA  \bm{v}+\bm{q} + \bm{D}_r(\SFB \bm{v}), \bm{k}(\bx)\rangle ~~~ \mathrm{s.t.} ~~~ \scalarp{\bm{v}_T{=}\bm{0},\bm{k}(\bx)}~~\text{on}~~(0,T) {\times} \mathbb{X},\label{eq:fINALtRICK} 
    \end{align}
\end{subequations}
    where \eqref{eq:fINALtRICK} is obtained by using the definitions \eqref{eq:Sampl}-\eqref{eq:Eembed}, the derivative reproducing property of Lemma \ref{lem::KernelMatrix} and the reproducing property in \eqref{eq:PreTrick} for $\SinjectIN k_{\bx} = \bm{k}(\bx) $ and
    $[\hat{B}\hat{V}](\bx)=\scalarp{\bm{B}\bm{v},\bm{k}(\bx)}$.
    \end{proof} 
\begin{algorithm}[ht!]
\small
    \caption{Infinitesimal Generator Kernel HJB Equation $(\texttt{IG-KHJB})$}
\label{alg:IG-kHJB}
\begin{algorithmic}[0]
\Require Data $\Set{D}^N$ \eqref{eq:data} \& samples for the state cost $\bm{q}_\SFX{=}[q(\bxi)]^N_{i=1}$, control penalty dual, diffusion $\epsilon {>} 0$, kernel $k$, regularizer $\gamma {>} 0$.
\State Compute $\bm{K} {\defeq} [k(\bxi,\bxj)]_{i,j\in[N]}$, $\bm{K}_\gamma {=} (\bm{K}{ + }N\gamma I_N)$ and define $(\mathcal{D}_r(\hat{B}\hat{V}))_{\SFX} {\defeq} [\mathcal{D}_r(\scalarp{{\bm{B}} \bm{v}{,}\bm{k}(\bx^{(i)})})]^N_{i=1}$
\State Compute $\SFA$ and $\SFB$ using \eqref{eq:FinDim} and $\bm{q} = \bm{K}_\gamma ^{-1}\bm{q}_\SFX$, $\bm{D}_r(\bm{B}\bm{v})\defeq\bm{K}_\gamma ^{-1}(\mathcal{D}_r(\hat{B}\hat{V}))_{\SFX}$
\Function{\textsc{hjb-fvp}}{$T,\bm{q},\SFA, \SFB, \bm{D}_r(\cdot)$}
\State Initialize $\bm{v}_T{=}\bm{0}$
\State Integrate  $-\dot{\bm{v}} =  \SFA  \bm{v}+\bm{q} + \bm{D}_r(\SFB \bm{v})$ from $T$ to $0$\Comment{e.g., using implicit Euler}
\State \Return $\bm{v}_0$
\EndFunction
\State $ \bm{v}_0$ = \textsc{hjb-fvp}($T,\bm{q},\SFA, \SFB, \bm{D}_r(\cdot)$)
\State Compute $\widehat{V}^\star_0(\bx)  
=  \scalarp{\bm{v}_0,\bm{k}(\bx)} \quad \text{and} \quad
\widehat{\bm{\pi}}^\star(\bx)  
\defeq \bm{u}^\star(\scalarp{{\bm{B}} \bm{v}_0,\bm{k}(\bx)})$
\end{algorithmic}
\end{algorithm}
\looseness=-1
\vspace{-0.5em}\section{Numerical Experiments}\label{sec:NumericalResults}
In this section, we present numerical examples to evaluate the performance of our \texttt{IG-KHJB} approach using Algorithm~\ref{alg:IG-kHJB}. We compare our approach to \cite{bevanda2024data} for optimal control of an unstable oscillator and to state-of-the-art NMPC for a swing-up and stabilization task on the inverted pendulum and cartpole systems. The latter employ $\texttt{Dojo}$ \citep{howelllecleach2022dojo} for dynamics simulation and $\texttt{Altro}$ \citep{howell2019altro} in a receding-horizon fashion (NMPC)\footnote{\texttt{ALTRO}'s parameters including initial guesses and OCP discretization were hand-tuned to best performance.}.
\\\textbf{Implementation details}
In the unstable oscillator example (\ref{subsec:unstable_oscillator}), we use the ODE of the system directly as well as analytical formulations for the kernel partial derivatives. In the examples using \texttt{Dojo}, the infinitesimal measurements for \eqref{eq:data} are not directly accessible and is therefore approximated via finite differences (FD). Additionally, to speed up computations, partial derivatives are also approximated by FD. 
We use an Euler-Implicit scheme to integrate \eqref{eq:finiteHJB} with a time step size $\Delta t$, with 
 parameter values for each experiment described in Table~\ref{tab:params}. The controllers for inverted pendulum and cartpole were smoothed using $\textstyle\frac{2\max|u|}{\pi}\arctan(\widehat{\pi}^\star(\bx))$ according to the input constraint from Table~\ref{tab:params}.
\vspace{-0.5em}
\subsection{Unstable Oscillator}\label{subsec:unstable_oscillator}
We compare our \texttt{IG-KHJB} to \cite{bevanda2024data} on the 2D Van der Pol Oscillator using identical dynamics and cost functions. The optimal infinite-horizon policy for $\epsilon \to 0^+$ can be analytically computed as $\boldsymbol{\pi}_{\infty}^\star(\bm{x}) = -x_1 x_2$ \citep{Villanueva2021}. Both methods are compared using the root mean square error (RMSE) against this optimal policy, evaluated on $1000$ uniformly sampled test points. Both methods employ an RBF kernel\footnote{The discrete-time kHJB requires a diffused RBF kernel; see \cite{bevanda2024data} for details.} $\mathrm{e}^{\textstyle -\nicefrac{\|\bm{x} -\bm{x}'\|^ 2}{\sigma^2}}$.
\begin{table}[t]
\centering
\scriptsize
\begin{tabular}{l|llllllll}
 \toprule
System  & Data grid & $\Set{X}_S$ & ${\max} |u|$ & $\sigma$ & $\epsilon$ & $\gamma$  & $\Delta t$ & $H$   \\
 \midrule
 Unstable Oscillator &  $ \sqrt{N} \times \sqrt{N}$  & $[\pm3, \pm3]$ & $-$ & $[5, 300]$ & $0.01$ & $10^{-8} $& $0.01$s & $1000$   \\
 Inverted Pendulum &  $50{\times}50$  &  $[\pm 0.99\pi, \pm 10]$ &$1.5$ Nm & $25$ & $0.02$ & $10^{-12} $& $0.02$s & ${500}$   \\
 Cartpole &  $ 9{\times}7{\times}23{\times}23$  &  $[\pm 2.5, \pm 3, \pm 0.99 \pi, \pm 8]$ & $7$ N &  $15$ & $0.01$ & $10^{-12}$& $0.01$s & $3000$ \\
\bottomrule
\end{tabular}
\vspace{-0.7em}
\caption{Parameter values used in Algorithm~\ref{alg:IG-kHJB} for the experiments in Sections \ref{subsec:InvPend} and \ref{subsec:Cartpole}. We use an implicit integration scheme to solve \texttt{IG-KHJB} over a time-horizon $T = H \Delta t $.}
\label{tab:params}
\end{table} 
In Figure~\ref{fig:tac_vs_us} (left), we evaluate the performance of \texttt{KHJB} \citep{bevanda2024data} against our proposed \texttt{IG-KHJB} for a varying lengthscale (Table \ref{tab:params}) using $N = 2500$ datapoints. In Figure~\ref{fig:tac_vs_us} (right), the optimal lengthscales are fixed to $\sigma_\texttt{KHJB} = 23, \sigma_\texttt{IG-KHJB} = 43$ and the RMSE is evaluated for both methods across $N \in \{5^2,\dots, 50^2\}$ data. The results show that our method is more robust to the kernel lengthscale choice and achieves up to three times lower RMSE with fewer data compared to \texttt{KHJB}. Notably, \texttt{IG-KHJB} achieves a lower RMSE than \texttt{KHJB} with only $N = 25$ datapoints.
\begin{figure}[t]
    \centering
    \begin{minipage}[t]{0.45\textwidth} 
        \centering
        \resizebox{\textwidth}{!}{
\usepgfplotslibrary{fillbetween}

\begin{tikzpicture}[/tikz/background rectangle/.style={fill={rgb,1:red,1.0;green,1.0;blue,1.0}, fill opacity={1.0}, draw opacity={1.0}}]
\begin{axis}[point meta max={nan}, point meta min={nan}, legend cell align={left}, legend columns={1}, title style={at={{(0.5,1)}}, anchor={south}, font={{\fontsize{14 pt}{18.2 pt}\selectfont}}, color={rgb,1:red,0.0;green,0.0;blue,0.0}, draw opacity={1.0}, rotate={0.0}, align={center}}, legend style={color={rgb,1:red,0.0;green,0.0;blue,0.0}, draw opacity={1.0}, line width={1}, solid, fill={rgb,1:red,1.0;green,1.0;blue,1.0}, fill opacity={1.0}, text opacity={1.0}, font={{\fontsize{14 pt}{18.2 pt}\selectfont}}, text={rgb,1:red,0.0;green,0.0;blue,0.0}, cells={anchor={center}}, at={(0.02, 0.02)}, anchor={south west}}, axis background/.style={fill={rgb,1:red,1.0;green,1.0;blue,1.0}, opacity={1.0}}, anchor={north west}, xshift={0.0mm}, yshift={-0.0mm}, width={120.0mm}, height={55.0mm}, scaled x ticks={false}, xlabel={lengthscale $\sigma$}, x tick style={color={rgb,1:red,0.0;green,0.0;blue,0.0}, opacity={1.0}}, x tick label style={color={rgb,1:red,0.0;green,0.0;blue,0.0}, opacity={1.0}, rotate={0}}, xlabel style={at={(ticklabel cs:0.5)}, anchor=near ticklabel, at={{(ticklabel cs:0.5)}}, anchor={near ticklabel}, font={{\fontsize{14 pt}{18.2 pt}\selectfont}}, color={rgb,1:red,0.0;green,0.0;blue,0.0}, draw opacity={1.0}, rotate={0.0}}, xmode={log}, log basis x={10}, xmajorgrids={true}, xmin={4.422068511161236}, xmax={339.20777034865506}, xticklabels={{$10^{1}$,$10^{2}$}}, xtick={{10.0,100.0}}, xtick align={inside}, xticklabel style={font={{\fontsize{14 pt}{18.2 pt}\selectfont}}, color={rgb,1:red,0.0;green,0.0;blue,0.0}, draw opacity={1.0}, rotate={0.0}}, x grid style={color={rgb,1:red,0.0;green,0.0;blue,0.0}, draw opacity={0.1}, line width={0.5}, solid}, axis x line*={left}, x axis line style={color={rgb,1:red,0.0;green,0.0;blue,0.0}, draw opacity={1.0}, line width={1}, solid}, scaled y ticks={false}, ylabel={RMSE to $\boldsymbol{\pi}_{\infty}^\star(\bm{x})$}, y tick style={color={rgb,1:red,0.0;green,0.0;blue,0.0}, opacity={1.0}}, y tick label style={color={rgb,1:red,0.0;green,0.0;blue,0.0}, opacity={1.0}, rotate={0}}, ylabel style={at={(ticklabel cs:0.5)}, anchor=near ticklabel, at={{(ticklabel cs:0.5)}}, anchor={near ticklabel}, font={{\fontsize{14 pt}{18.2 pt}\selectfont}}, color={rgb,1:red,0.0;green,0.0;blue,0.0}, draw opacity={1.0}, rotate={0.0}}, ymode={log}, log basis y={10}, ymajorgrids={true}, ymin={0.012652728631728559}, ymax={2.636413934340885}, yticklabels={{$10^{-1}$,$10^{0}$}}, ytick={{0.1,1.0}}, ytick align={inside}, yticklabel style={font={{\fontsize{14 pt}{18.2 pt}\selectfont}}, color={rgb,1:red,0.0;green,0.0;blue,0.0}, draw opacity={1.0}, rotate={0.0}}, y grid style={color={rgb,1:red,0.0;green,0.0;blue,0.0}, draw opacity={0.1}, line width={0.5}, solid}, axis y line*={left}, y axis line style={color={rgb,1:red,0.0;green,0.0;blue,0.0}, draw opacity={1.0}, line width={1}, solid}, colorbar={false}]
    \addplot[color={rgb,1:red,1.0;green,0.0;blue,0.0}, name path={13}, draw opacity={1.0}, line width={2}, solid, mark={*}, mark size={3.00 pt}, mark repeat={1}, mark options={color={rgb,1:red,1.0;green,0.0;blue,0.0}, draw opacity={1.0}, fill={rgb,1:red,1.0;green,0.0;blue,0.0}, fill opacity={1.0}, line width={0.75}, rotate={0}, solid}]
        table[row sep={\\}]
        {
            \\
            5.0  0.8914830371894003  \\
            6.0  0.8343934961721745  \\
            8.0  0.8685916574897006  \\
            10.0  0.8982823699932485  \\
            12.0  0.12872798359405982  \\
            15.0  0.11978526388088791  \\
            18.0  0.06694130113587994  \\
            23.0  0.05555018906551025  \\
            28.0  0.05768411360375936  \\
            35.0  0.0746996496733067  \\
            43.0  0.13579850306857047  \\
            54.0  0.3230492156358038  \\
            66.0  0.6871678420655413  \\
            82.0  1.3853200280065616  \\
            102.0  1.9125703528174316  \\
            127.0  2.1393857597214643  \\
            157.0  2.236731979827714  \\
            195.0  2.261814893181777  \\
            242.0  2.26610458598965  \\
            300.0  2.2666598455973825  \\
        }
        ;
    \addlegendentry {KHJB}
    \addplot[color={rgb,1:red,0.0;green,0.0;blue,1.0}, name path={14}, draw opacity={1.0}, line width={2}, solid, mark={*}, mark size={3.00 pt}, mark repeat={1}, mark options={color={rgb,1:red,0.0;green,0.0;blue,1.0}, draw opacity={1.0}, fill={rgb,1:red,0.0;green,0.0;blue,1.0}, fill opacity={1.0}, line width={0.75}, rotate={0}, solid}]
        table[row sep={\\}]
        {
            \\
            5.0  0.8705458960143422  \\
            6.0  0.8681049429640962  \\
            8.0  0.9122850225405008  \\
            10.0  0.8573027648451294  \\
            12.0  0.8528364194386885  \\
            15.0  0.8537443527202663  \\
            18.0  0.8031956436976372  \\
            23.0  0.29418628848469397  \\
            28.0  0.12420997677874408  \\
            35.0  0.028543248043345903  \\
            43.0  0.014716734024699477  \\
            54.0  0.016889143468640243  \\
            66.0  0.018929504207262825  \\
            82.0  0.020286268113793447  \\
            102.0  0.0206955013994289  \\
            127.0  0.020187544803690095  \\
            157.0  0.020636141230991473  \\
            195.0  0.03848949861187057  \\
            242.0  0.1531151725547176  \\
            300.0  0.6588486359969794  \\
        }
        ;
    \addlegendentry {IG-KHJB (ours)}
\end{axis}
\end{tikzpicture}}
        \label{fig:first_plot}
    \end{minipage}
   \hspace{0.05\linewidth} 
    \begin{minipage}[t]{0.45\textwidth} 
        \centering
        \resizebox{\textwidth}{!}{

\begin{tikzpicture}[/tikz/background rectangle/.style={fill={rgb,1:red,1.0;green,1.0;blue,1.0}, fill opacity={1.0}, draw opacity={1.0}}]
\begin{axis}[point meta max={nan}, point meta min={nan}, legend cell align={left}, legend columns={1}, title style={at={{(0.5,1)}}, anchor={south}, font={{\fontsize{14 pt}{18.2 pt}\selectfont}}, color={rgb,1:red,0.0;green,0.0;blue,0.0}, draw opacity={1.0}, rotate={0.0}, align={center}}, legend style={color={rgb,1:red,0.0;green,0.0;blue,0.0}, draw opacity={1.0}, line width={1}, solid, fill={rgb,1:red,1.0;green,1.0;blue,1.0}, fill opacity={1.0}, text opacity={1.0}, font={{\fontsize{14 pt}{18.2 pt}\selectfont}}, text={rgb,1:red,0.0;green,0.0;blue,0.0}, cells={anchor={center}}, at={(0.02, 0.78)}, anchor={north west}}, axis background/.style={fill={rgb,1:red,1.0;green,1.0;blue,1.0}, opacity={1.0}}, anchor={north west}, xshift={0.0mm}, yshift={-0.0mm}, width={120.0mm}, height={55.0mm}, scaled x ticks={false}, xlabel={dataset size N}, x tick style={color={rgb,1:red,0.0;green,0.0;blue,0.0}, opacity={1.0}}, x tick label style={color={rgb,1:red,0.0;green,0.0;blue,0.0}, opacity={1.0}, rotate={0}}, xlabel style={at={(ticklabel cs:0.5)}, anchor=near ticklabel, at={{(ticklabel cs:0.5)}}, anchor={near ticklabel}, font={{\fontsize{14 pt}{18.2 pt}\selectfont}}, color={rgb,1:red,0.0;green,0.0;blue,0.0}, draw opacity={1.0}, rotate={0.0}}, xmode={log}, log basis x={10}, xmajorgrids={true}, xmin={21.774089748902018}, xmax={2870.3840537422034}, xticklabels={{$10^{2}$,$10^{3}$}}, xtick={{100.0,1000.0}}, xtick align={inside}, xticklabel style={font={{\fontsize{14 pt}{18.2 pt}\selectfont}}, color={rgb,1:red,0.0;green,0.0;blue,0.0}, draw opacity={1.0}, rotate={0.0}}, x grid style={color={rgb,1:red,0.0;green,0.0;blue,0.0}, draw opacity={0.1}, line width={0.5}, solid}, axis x line*={left}, x axis line style={color={rgb,1:red,0.0;green,0.0;blue,0.0}, draw opacity={1.0}, line width={1}, solid}, scaled y ticks={false}, y tick style={color={rgb,1:red,0.0;green,0.0;blue,0.0}, opacity={1.0}}, y tick label style={color={rgb,1:red,0.0;green,0.0;blue,0.0}, opacity={1.0}, rotate={0}}, ylabel style={at={(ticklabel cs:0.5)}, anchor=near ticklabel, at={{(ticklabel cs:0.5)}}, anchor={near ticklabel}, font={{\fontsize{14 pt}{18.2 pt}\selectfont}}, color={rgb,1:red,0.0;green,0.0;blue,0.0}, draw opacity={1.0}, rotate={0.0}}, ymode={log}, log basis y={10}, ymajorgrids={true}, ymin={0.011731559498159195}, ymax={2.9203362085329556}, yticklabels={{$10^{-1}$,$10^{0}$}}, ytick={{0.1,1.0}}, ytick align={inside}, yticklabel style={font={{\fontsize{14 pt}{18.2 pt}\selectfont}}, color={rgb,1:red,0.0;green,0.0;blue,0.0}, draw opacity={1.0}, rotate={0.0}}, y grid style={color={rgb,1:red,0.0;green,0.0;blue,0.0}, draw opacity={0.1}, line width={0.5}, solid}, axis y line*={left}, y axis line style={color={rgb,1:red,0.0;green,0.0;blue,0.0}, draw opacity={1.0}, line width={1}, solid}, colorbar={false}]
    \addplot[color={rgb,1:red,1.0;green,0.0;blue,0.0}, name path={21}, draw opacity={1.0}, line width={2}, solid, mark={*}, mark size={3.00 pt}, mark repeat={1}, mark options={color={rgb,1:red,1.0;green,0.0;blue,0.0}, draw opacity={1.0}, fill={rgb,1:red,1.0;green,0.0;blue,0.0}, fill opacity={1.0}, line width={0.75}, rotate={0}, solid}]
        table[row sep={\\}]
        {
            \\
            25.0  2.4981548227630004  \\
            36.0  2.141666845115811  \\
            49.0  2.133981173104006  \\
            64.0  2.143948718878298  \\
            100.0  2.140182497860756  \\
            121.0  2.1405738059843094  \\
            169.0  2.1389676490315375  \\
            256.0  2.140230227679776  \\
            361.0  0.3668125110824574  \\
            484.0  2.1392501372595496  \\
            676.0  2.1384752322500513  \\
            961.0  0.05063332733674794  \\
            1296.0  0.11527462678334903  \\
            1764.0  0.0697346731534538  \\
            2500.0  0.05681645573152186  \\
        }
        ;
    \addlegendentry {KHJB}
    \addplot[color={rgb,1:red,0.0;green,0.0;blue,1.0}, name path={22}, draw opacity={1.0}, line width={2}, solid, mark={*}, mark size={3.00 pt}, mark repeat={1}, mark options={color={rgb,1:red,0.0;green,0.0;blue,1.0}, draw opacity={1.0}, fill={rgb,1:red,0.0;green,0.0;blue,1.0}, fill opacity={1.0}, line width={0.75}, rotate={0}, solid}]
        table[row sep={\\}]
        {
            \\
            25.0  0.029958372720903622  \\
            36.0  0.02574742006404572  \\
            49.0  0.022910017936893345  \\
            64.0  0.02095149633432863  \\
            100.0  0.018528018588036464  \\
            121.0  0.0177376663466407  \\
            169.0  0.01663907549080507  \\
            256.0  0.015652600138204926  \\
            361.0  0.015066846579782371  \\
            484.0  0.01469003492717706  \\
            676.0  0.014365581428486741  \\
            961.0  0.014111933161482938  \\
            1296.0  0.013953684918873022  \\
            1764.0  0.013825947768375674  \\
            2500.0  0.013714161217254252  \\
        }
        ;
    \addlegendentry {IG-KHJB (ours)}
\end{axis}
\end{tikzpicture}}
        \label{fig:second_plot}
    \end{minipage}
    \vspace{-20pt}
    \caption{Comparison of RMSE to the known optimal policy $\boldsymbol{\pi}_{\infty}^\star(\bm{x})$ between \texttt{KHJB} ( \cite{bevanda2024data}) and our \texttt{IG-KHJB} approach {for the Van der Pol Oscillator}.}
    \label{fig:tac_vs_us}
\end{figure}
\vspace{-0.5em}
\subsection{Torque Limited Inverted Pendulum} 
\label{subsec:InvPend}
\begin{wrapfigure}{R}{0.6\textwidth}
    \centering
    \includegraphics[trim={0cm 0.1cm 0cm 0.1cm},clip, width=0.99\linewidth]{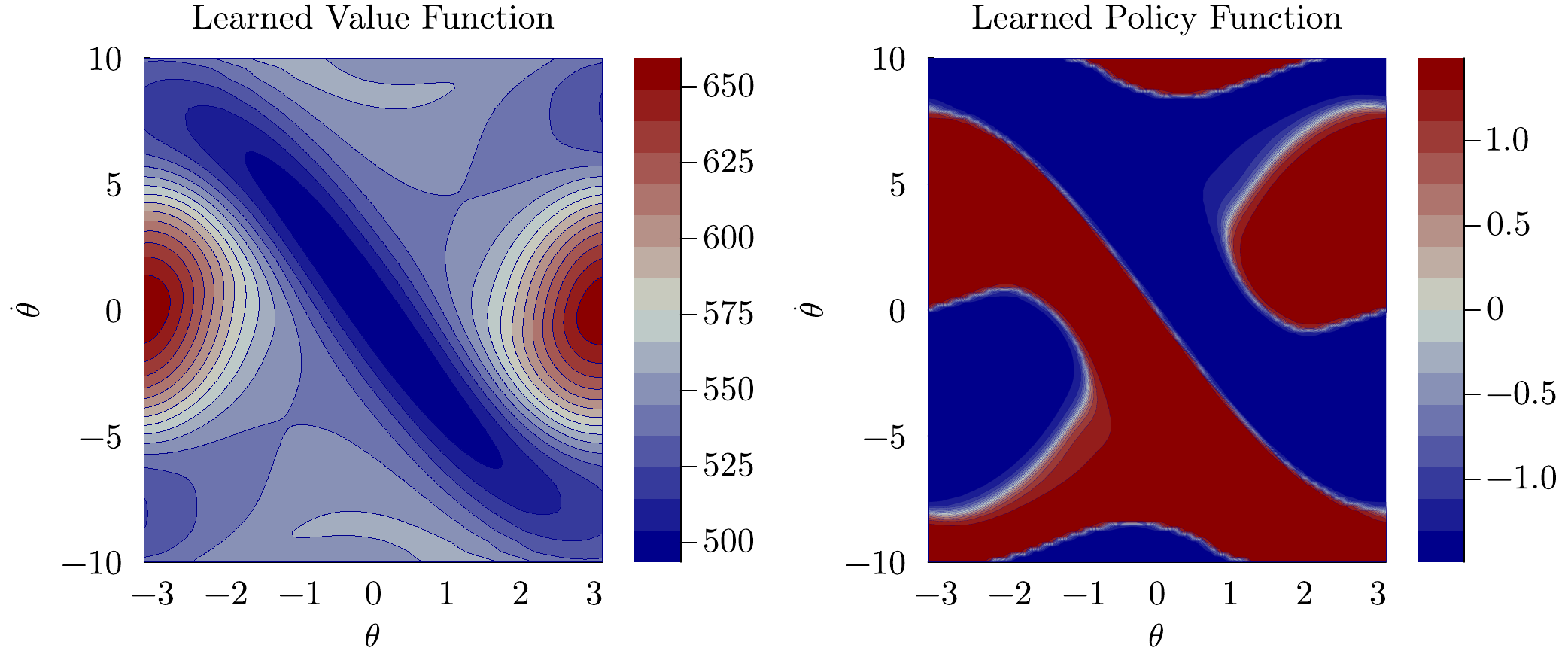}
    \vspace{-0.5em}
    \caption{Contour plots of the value and controller functions learned from $2.5\cdot10^3$ samples using our \texttt{IG-KHJB} approach (upright at $(\dot{\theta},\theta){=}(0,0)$).}
    \label{fig:pendulum_V_pi}
\end{wrapfigure}
\looseness = -1
We evaluate our method on a swing-up and stabilization task at the upright equilibrium $\theta=0$ for the inverted pendulum. The value function and policy shown in Fig.~\ref{fig:pendulum_V_pi} are learned using Algorithm \ref{alg:IG-kHJB} with data \eqref{eq:data} generated from $\texttt{Dojo}$ \citep{howelllecleach2022dojo}. We use a Laplace kernel $\mathrm{e}^{\textstyle -\nicefrac{\|\bm{x} -\bm{x}'\|}{\sigma}}$, that is additionally smoothed to allow for stable derivative computations\footnote{We provide the derivative at $\bm{0}$ using those of an RBF kernel with $\nicefrac{\sigma}{100}$.}. The employed parameters can be found in Table~\ref{tab:params}.
To deal with wrap-around of $\theta$ we define the state $\bm{x} \defeq [c_{\theta}, s_{\theta}, \dot{\theta}]^\top$ using shorthands $s_{\theta}\defeq \sin \theta$ and $c_{\theta}\defeq \cos \theta$. The state cost is given by $q(\bm{x}) =  q_1 s_{\theta}^2 + q_2 (c_{\theta} - 1)^2 + q_v \dot \theta ^2 $ with $q_1, q_2 = 30, q_v = 1$ and control penalization $r(u) = \textstyle \|u\|^2_{\nicefrac{1}{2}}$, clipped at the control limits. 
We compare our policy with regard to the accumulated trajectory costs to state-of-the-art NMPC using $\texttt{Altro}$ \citep{howell2019altro} for trajectory optimization. We use the same pendulum dynamics for $\texttt{Altro}$ as in \texttt{Dojo}\footnote{With $m=1.0 \mathrm{kg}$, $g=9.81\mathrm{m/s}^2$, {$l=1.0 \mathrm{m}$, } $b=0.1 \mathrm{kg.m}^2\mathrm{s}^{-1}$, $I = 0.0842 \mathrm{kg.m}^2$ with $m, g, l, b$ being respectively the mass, the gravitational constant and length of the pendulum, viscous damping and inertia around its mass center.}.
Both our policy, $\widehat{\bm{\pi}}^\star(\bx)$, and $\texttt{Altro-NMPC}$ are deployed with a control frequency of $50 \mathrm{Hz}$. Figure~\ref{fig:costs_invPend} shows that our method achieves lower mean accumulated costs than $\texttt{Altro-NMPC}$ under identical stage costs. Each lasting 5 seconds, the trajectories were simulated with initial states sampled uniformly from $[\pm \pi, \pm 8]$. The mean accumulated costs are averaged over $10$ runs, with $50$ different initial positions each.
\begin{figure}[t]
    \centering
    \caption{Accumulated stage costs using our learned policy $\widehat{\bm{\pi}}^\star(\bx)$ and $\texttt{Altro-NMPC}$.}
    \subfigure[Inverted Pendulum]{ 
        \centering
        \resizebox {.48\textwidth} {!} {\input{figures/pendulum/25-04-20_Accumulated_costs_plot_pgfplots_weightspace}}
        \label{fig:costs_invPend}
        }
    \hfill 
    \subfigure[Cartpole]{ 
        \centering
        \resizebox {.48\textwidth} {!} {\input{figures/cartpole/Accumulated_costs_cartpole_pgfplots}}
        \label{fig:costs_Cartpole}
        }
    \vspace{-0.5cm}
    \label{fig:costs}
\end{figure}
\vspace{-0.5em}
\subsection{Cartpole: Inverted Pendulum on a Cart}
\label{subsec:Cartpole}
We now consider the swing-up and stabilization task for a Cartpole system and compare it to $\texttt{Altro-NMPC}$. We again use the smoothed Laplace kernel, with the parameters in Table~\ref{tab:params} and data from \texttt{Dojo}.
Here, we also augment the state space representation to $\bm{x} = [x, \dot x, c_{\theta}, s_{\theta}, \dot \theta] ^\top$. To describe the task via stage costs, we penalize the Euclidean distance between the pendulum’s endpoint $(x_p,y_p)$ and the goal $(0,l)$ \citep{mcallister2017data}, defined as $d(x, c_\theta, s_\theta)\defeq q_hx^2_p+ q_v(y_p-l)^2=q_h(x - l s_{\theta})^2 +  q_v(c_{\theta}-1)^2$, with a velocity penalty, this amounts to the stage cost $q(\bm{x}) =  d(x, c_\theta, s_\theta) + q_{vel} \dot x ^2 + q_\omega \dot \theta ^2$ and control penalty $r(u) = \textstyle \|u\|^2_{\nicefrac{1}{5}}$, where $q_h = 10$, $q_v = 100l^2$, $q_{vel},q_\omega = 1$. 
For the cost comparison, trajectories with a duration of $10 \mathrm{s}$ are simulated with initial states sampled random uniformly on $[0, \pm 2, \pm \pi, \pm 6]$. Here, our controller's accumulated costs (see Figure \ref{fig:costs_Cartpole}) are lower than those of \texttt{Altro-NMPC}\footnote{The same dynamics as in \texttt{Dojo} are used, i.e. $M = 0.5 \mathrm{kg}$, $m = 0.5 \mathrm{kg}$, $l = 1.0 \mathrm{m}$, $b=0.05\mathrm{kg.m}^2\mathrm{s}^{-1}$, $k = 0.05 \mathrm{kg.s}^{-1}$, $I = 0.0513 \mathrm{kg.m}^2$, being respectively the cart mass, pendulum mass, pendulum length, viscous rotational damping, viscous damping of the cart, and the pendulum inertia around the center of mass.} by a large margin, computed over 10 runs with 50 different initial states per run and deployed at $200 \mathrm{Hz}$. To eliminate confounding effects from model‐mismatch, we ran \texttt{Altro} on an analytic dynamics model that we built to replicate the assumptions and parameterization used by \texttt{Dojo} in our experiments.
\vspace{-0.5em}
\section{Conclusion}
This article has introduced novel methods for data-driven nonlinear system identification and globally optimal stochastic control. In detail, we have derived non-parametric estimators of infinitesimal generators of optimally controlled diffusions, as summarized in Lemma~\ref{lem::KernelMatrix}. Moreover, in Proposition~\ref{prop::HJB} we have formulated a continuous-time Kernel Hamilton-Jacobi-Bellman (KHJB) equation, which enables the computation of data-driven approximations of globally optimal solutions to stochastic optimal control problems. Our method, outlined in Algorithm \ref{alg:IG-kHJB}, has been demonstrated to outperform modern data-driven and classical nonlinear programming methods for optimal control in both synthetic and robotics benchmarks.

\newpage
\acks{We thank Robert Lefringhausen and Max Beier for their valuable feedback while preparing this manuscript. This work was supported by the European Union’s Horizon Europe
innovation action programme under grant agreement No. 101093822,
“SeaClear2.0”, the DAAD programme Konrad Zuse Schools of Excellence in Artificial Intelligence, sponsored by the Federal Ministry of Education and Research.}

\bibliography{references}

\end{document}